\documentclass[12pt]{amsart}
\usepackage{amsfonts,amssymb, amscd, latexsym, graphicx, psfrag, color, float}
\usepackage[all]{xy}

\newtheorem{dummy}{dummy}[section]
\newtheorem{lemma}[dummy]{Lemma}
\newtheorem{theorem}[dummy]{Theorem}

\newtheorem{conjecture}[dummy]{Conjecture}
\newtheorem{corollary}[dummy]{Corollary}

\theoremstyle{definition}
\newtheorem{definition}[dummy]{Definition}
\newtheorem{example}[dummy]{Example}
\newtheorem{remark}[dummy]{Remark}

\newcommand{\bC}{\mathbb{C}}

\newcommand{\bN}{\mathbb{N}}
\newcommand{\bP}{\mathbb{P}}

\newcommand{\bR}{\mathbb{R}}
\newcommand{\bZ}{\mathbb{Z}}

\newcommand{\cC}{\mathcal{C}}
\newcommand{\cE}{\mathcal{E}}
\newcommand{\cF}{\mathcal{F}}

\newcommand{\cP}{\mathcal{P}}

\newcommand{\cR}{\mathcal{R}}

\newcommand{\Sh}{\mathit{Sh}}

\newcommand{\Rep}{\mathit{Rep}}
\newcommand{\Hom}{\mathrm{Hom}}

\newcommand{\Perf}{\mathcal{P}\mathrm{erf}}

\newcommand{\Cone}{\mathsf{C}one}

\renewcommand{\SS}{\mathit{SS}}

\newcommand{\cpm}{\mathrm{CPM}}
\newcommand{\CPM}{\mathrm{CPM}}

\renewcommand{\SS}{\mathit{SS}}

\newcommand{\Chord}{\mathrm{Chord}}

\newcommand{\dgCat}{\mathrm{dg}\mathcal{C}\mathrm{at}}

\newcommand{\Cmod}{\bC\text{-mod}}

\newcommand*{\triple}[2][.1ex]{%
  \mathrel{\vcenter{\offinterlineskip%
  \hbox{$#2$}\vskip#1\hbox{$#2$}\vskip#1\hbox{$#2$}}}} 

\newcommand*{\triplerightarrow}{\triple{\rightarrow}}

\begin{document}

\title[Mirror Symmetry in dimension  one and Fourier-Mukai equivalences]{Mirror Symmetry in dimension  one and Fourier-Mukai equivalences}

\begin{abstract}
In this paper we will describe an approach to mirror symmetry for appropriate $1$-dimensional DM stacks of arithmetic genus $g \leq 1$, called \emph{tcnc} curves, which was developed by the author with Treumann and Zaslow in \cite{STZ}. This involves introducing a conjectural sheaf-theoretic model for the Fukaya category of punctured Riemann surfaces. As an application, we will investigate derived equivalences of \emph{tcnc} curves, and generalize classic results of Mukai on dual abelian varieties \cite{M}. 
\end{abstract}

\author{Nicol\`o Sibilla}
\address{Nicol\`o Sibilla, Max Planck Institute for Mathematics, Vivatsgasse 7,
53111 Bonn,
Germany}
\email{sibilla@mpim-bonn.mpg.de}

\maketitle

{\small \tableofcontents}

\section{Introduction}

As originally formulated by Kontsevich \cite{K}, Homological Mirror Symmetry (from now on, HMS) relates the derived category of coherent sheaves on a Calabi-Yau variety $X$, $D^b(Coh(X))$, and the Fukaya category of a symplectic manifold $\hat{X}$, by stating that if $X$ and $\hat{X}$ are mirror partners, then $D^b(Coh(X)) \cong Fuk(\hat{X})$. Since its proposal, much work has been done towards establishing Kontsevich's conjecture in important classes of examples, see \cite{PZ, S1, Sh}, and references therein. 

One of the main obstacles for tackling Kontsevich's conjecture is gaining a sufficient understanding of the Fukaya category.\footnote{For foundational material on the Fukaya category, the reader should consult \cite{FOOO}, and \cite{S}.} Starting in 2009, in various talks, Kontsevich has argued \cite{K1} that the Fukaya category of a Stein manifold should have good local-to-global properties, and therefore conjecturally could be recovered as the global sections of a suitable sheaf of dg categories.\footnote{For a sampling of some of the relevant work in this direction, see also \cite{S2,S3,N1,NT}.} This is in keeping with previous work of Nadler and Zaslow who, in \cite{NZ} and \cite{N}, establish an equivalence between the Fukaya category of exact Lagrangians in a cotangent bundle $T^*X$, and the dg category of complexes of cohomologically constructible sheaves over $X$, $Sh(X)$.\footnote{From now on, we will refer to objects in $Sh(X)$ simply as `constructible sheaves.' For a comprehensive introduction to the subject we refer the reader to \cite{KS}.}  
 
Following Kontsevich's insight, in \cite{STZ}, joint with Treumann and Zaslow, we equip the Lagrangian skeleton of a punctured Riemann surface $\Sigma$ with a sheaf of dg categories, called $\cpm(-)$,\footnote{$\cpm$ stands for `constructible plumbing model,' as this framework can be applied  more generally to investigate the Fukaya category of a plumbing of cotangent bundles, for which see also \cite{A}.} such that its local behavior is dictated by Nadler and Zaslow's work on cotangent bundles, while its global sections are conjecturally quasi-equivalent to the Fukaya category of compact exact Lagrangians in $\Sigma$, $Fuk(\Sigma)$. Further, in \cite{STZ}, using this model as a stand in for the Fukaya category, we prove a version of HMS which pairs suitable stacky, degenerate elliptic curves, called \emph{tcnc} curves (see Section \ref{sec:tcnc}), and punctured symplectic tori. 

In the first part of this paper, we give a quick review of the results contained in \cite{STZ}, with a special emphasis on motivations and examples.
In Section \ref{sec:cpm}, after introducing the necessary background, we define $\cpm(-)$ as a sheaf of dg categories on a suitable Grothendieck site of decorated ribbon graphs, and open inclusions. The applications to mirror symmetry are explained in Section \ref{sec:hms}. Given a tcnc curve $C$, we explain how to construct a ribbon graph $D_{\hat{C}}$, which arises as the skeleton of a punctured symplectic torus $\hat{C}$, and we prove that there is an equivalence $\Perf(C) \cong \cpm(D_{\hat{C}})$. Granting the conjectural equivalence $\cpm(D_{\hat{C}}) \cong Fuk(\hat{C})$, we obtain a HMS statement relating $C$ and $\hat{C}$.

The HMS statement proved in \cite{STZ} can be used to explore the algebraic geometry of tcnc curves. In Section \ref{sec:deq} we prove that, up to derived equivalence, tcnc curves are classified by the sum of the orders of the isotropy groups at the nodes.  This generalizes work of Mukai on derived auto-equivalences of smooth elliptic curves \cite{M}, and of Burban and Kreussler who considered the case of the nodal $\bP^1$ \cite{BK1}.
From the standpoint of mirror symmetry, this result corresponds to the simple fact that the Fukaya category of a punctured Riemann surface depends exclusively on genus, and number of punctures. 

{\bf Acknowledgments:} I wish to thank the organizers of ``Mirror Symmetry and Tropical Geometry'' for making possible this very stimulating event. 
It is a pleasure to thank David Treumann and Eric Zaslow for numerous conversations about the results described here, and for our collaboration \cite{STZ}, which is the starting point of this project. 
 
\section{A model for the Fukaya category of punctured Riemann surfaces}
\label{sec:cpm}
In this section we review the construction of $\cpm(-)$. We will follow closely the exposition of \cite{STZ}, but we shall gloss over many technical aspects of the theory, for which we refer the reader to the original paper.
Section \ref{sec:mic} contains a brief overview of definitions and results from microlocal sheaf theory which will be needed later, and a preliminary, `local,' definition of $\cpm(-)$. Section \ref{sec:quiv} discusses a useful dictionary between category of sheaves, and categories of quiver representations. In Section \ref{sec:chord} we introduce the notion of \emph{chordal} ribbon graph, and give the full definition of $\cpm(-)$, as a sheaf of dg categories over the Grothendieck site of chordal ribbon graphs. 

Before proceeding, it is convenient to clarify what we mean by sheaf of dg categories. Recall that, following Tabuada \cite{Tab}, we can equip the category of small dg categories,
$\dgCat$, with a model structure. For us, a sheaf on a site $\cC$ with values in a model category $\mathcal{D}$ is a pre-sheaf $F$, such that, whenever $S = \{U_i\}$ is a covering sieve for $U \in \cC$, the diagram 
$$
F(U) \rightarrow [\Pi_i F(U_i) \rightrightarrows \Pi_{i,j} F(U_i \times_U U_j)  \triplerightarrow \dots]
$$ is a \emph{homotopy} limit in $\mathcal{D}$. The sheaf property can be verified in practice quite easily, using the following description of equalizers in $\dgCat$.
\begin{lemma}
\label{lem:ho eq}
Let $\xymatrix{\ \mathcal{C} \ar@<1ex>[r]^{F} \ar@<-0.5ex>[r]_{G}& \mathcal{C'}}$ be a diagram in $dgCat$, and denote $\cE$ the dg category having
\begin{itemize}
\item as objects, pairs $(C, u)$, where $C \in \mathcal{C}$, and $u: F(C) \rightarrow G(C)$ is a degree zero, closed morphism 
, which becomes invertible in the homotopy category,
\item as morphisms, pairs $(f, H) \in hom^k(C, C')\oplus hom^{k-1}(F(C), G(C'))$, with differential given by $d(f, H) = (df, dH - (u'F(f) - G(f) u))$. The composition is obvious.
\end{itemize}
Then $\cE$, equipped with the forgetful functor $\cE \rightarrow \mathcal{C}$, is a homotopy equalizer for $F$ and $G$.
\end{lemma}
\begin{proof}
Lemma \ref{lem:ho eq} depends on the availabilty of an explicit construction of the \emph{path object} $P(\mathcal{C'})$ for $\mathcal{C'}$, which can be found in Lemma 4.1 of \cite{Tab1}. This allows us to compute the homotopy equalizer in the usual way, by taking appropriate fibrant replacements. We leave the details to the reader.
\end{proof} 

\subsection{Microlocal sheaf theory in dimension $1$}
\label{sec:mic}
Let $X$ be a manifold, and let $Sh(X)$ be the category of constructible sheaves over $X$.
In \cite{KS}, Kashiwara and Schapira explain how to attach to a constructible sheaf $\cF \in Sh(X)$ a \emph{conical} (i.e. invariant under fiberwise dilation) Lagrangian subset of $T^*X$, called \emph{singular support}, and denoted $SS(\cF)$. Informally, $SS(\cF)$ is an invariant encoding the co-directions along which $\cF$ does not `propagate.' Rather than giving the general definition, for which we refer the reader to Section 5.1 of \cite{KS}, we will describe the singular support in the simpler set up which will be needed in the following. 

Assume that $X$ is a $1$-dimensional manifold equipped with affine structure. Let $x$ be a point of $X$, and let $f$ be an affine $\bR$-valued function on $X$ around $x$.  For $\epsilon > 0$ sufficiently small let $A:=\{y \in X \mid f(y) < f(x) + \epsilon\}$, and $B:=\{y \in X \mid f(y) < f(x) - \epsilon\}$.  We define the functor $\mu_{x,f}:\Sh(X) \to \bC\text{-mod}$ to be the cone of the restriction map
$\Gamma(A;F\vert_A) \to \Gamma(B;F\vert_B).$

Since every constructible sheaf $F$ is locally constant in a deleted neighborhood of $x$, this functor does not depend on $\epsilon$, if $\epsilon$ is small enough.
Also, $\mu_{x,f}$ depends only on $x$ and $df_x$.  When $(x,\xi) \in T^*X$ we let $\mu_{x,\xi}$ denote the functor associated to the point $x$ and the affine function whose derivative at $x$ is $\xi$.

\begin{definition}
For each $F \in \Sh(X)$ we define $\SS(F) \subset T^*X$, the \emph{singular support} of $F$, to be the closure of the set of all $(x,\xi) \in T^*X$ such that $\mu_{x,\xi} F \neq 0$.
\end{definition}

Note thate, as $\mu_{x,\xi} = \mu_{x,t\cdot \xi}$ when $t > 0$, the set $\SS(F)$ is conical. In fact, if $(x,\xi) \in \SS(F)$ and $t \in \bR_{>0}$, then $(x,t \cdot \xi) \in \SS(F)$. Further, $\SS(F)$ is 1-dimensional and therefore a Lagrangian subset of $T^*X$ with its usual symplectic form.

\begin{definition}
Suppose $\Lambda \subset T^*X$ is a conical Lagrangian.  Define $\Sh(X,\Lambda) \subset \Sh(X)$ to be the full triangulated subcategory of sheaves with $\SS(F) \subset \Lambda$.
\end{definition}

\begin{example}
Let $\Lambda = X \cup T_{s_1}^*X \cup \cdots \cup T_{s_n}^* X$ be the union of the zero section and the cotangent spaces of finitely many points $\{s_1,\ldots,s_n\}$.  Then $\Sh(X,\Lambda)$ is the category of sheaves that are locally constant away from $\{s_1,\ldots,s_n\}$.
\end{example}

If $F$ is a sheaf  in $\Sh(X, \Lambda)$, and $\xi \neq 0$, $\mu_{x,\xi}(F) \in \bC\text{-mod}$ should be thought of as the (microlocal) `stalk' of $F$ over $(x, \xi) \in \Lambda \setminus X$. This suggests that sheaves with singular support in $\Lambda$ have a local nature over $\Lambda$, as well as over $X$. The locality of $\Sh(X, \Lambda)$ over $X$ can be encoded in the claim that the assignment: 
$
U \subset^{open} X \mapsto \Sh_{\Lambda}(U) := \Sh(U, T^*U \cap \Lambda),
$ 
defines a sheaf of dg categories over $X$. In Definition \ref{def:cpm1} we will introduce a sheaf of dg categories, denoted $\cpm(-)$, which, in an appropriate sense, is an extension of $\Sh_{\Lambda}(-)$ to $\Lambda$. In particular, we will have $\cpm(\Lambda) \cong \Sh(X, \Lambda)$.

Let $U \subset T^*X$ be an open subset, and let $\cP(X, U)$ be the Verdier quotient of $\Sh(X)$ by the thick subcategory of all sheaves $F$ with $\SS(F) \cap U = \varnothing$ (see \cite{KS}, Section 6.1). Consider the full subcategory of $\cP(X, U)$ spanned by sheaves $F$ such that $SS(F) \cap U \subset \Lambda \cap U$, and denote it $\cP_{\Lambda}(X, U)$. Both $\cP(X, -)$, and $\cP_{\Lambda}(X,-)$, naturally define pre-sheaves of dg categories on $T^*X$. We can therefore consider the sheafification of $\cP_{\Lambda}(X, -)$ over $T^*X$, which we denote $MSh_{\Lambda}(-)$.\footnote{Note that, if $\pi:T^*X \rightarrow X$ is the natural projection, then $\pi_*MSh(-) \cong \Sh_{\Lambda}(-)$.}
\begin{definition}
\label{def:cpm1}
Define $\cpm(-)$ to be the sheaf of dg categories over $\Lambda$ obtained by pulling back $MSh(-)$ along $i$, $\cpm(-) \cong i^*MSh(-)$. 
\end{definition}

\subsection{Microlocal sheaves and quiver representations}
\label{sec:quiv}
Assume that $X$ is a $1$-dimensional manifold and $\Lambda \hookrightarrow T^*X$ is a conical Lagrangian subset. 
The category $\Sh(X, \Lambda)$, and the sheaf $\cpm(-)$ over $\Lambda$, can be described very explicitly in terms of quiver representations.
  
Let us call the connected components of $\Lambda - X$ the \emph{spokes} of $\Lambda$.  They are divided into two groups depending on which component of $T^*X - X$ they fall into.  Using an orientation of $X$ we may label these groups ``upward'' and ``downward.''
The conic Lagrangian $\Lambda$ determines a partition $P_\Lambda$ of $X$ into subintervals (which may be open, half-open, or closed) and points.  Let us describe this partition in the case $X = \bR$, the general case is similar.  Each spoke of $\Lambda$ is incident with a point $x \in \bR$, which we may order $x_1 < \ldots < x_k$.  We put $\{x_i\} \in P_\Lambda$ if $x_i$ is incident with both an upward and a downward spoke.  We put an interval $I$ from $x_i$ to $x_{i+1}$ in $\cP_\Lambda$ whose boundary conditions are determined by the following rules
\begin{itemize}
\item If $x_i$ is incident with an upward spoke but not incident with a downward spoke, then $x_i$ is included in $I$.  Otherwise $x_i$ is not included in $I$.
\item If $x_{i+1}$ is incident with a downward spoke but not incident with an upward spoke, then $x_{i+1}$ is included in $I$.  Otherwise $x_{i+1}$ is not included in $I$.
\end{itemize}
We put $(-\infty,x_1)$ in $P_\Lambda$ if $x_1$ is incident with an upward spoke and $(-\infty,x_1]$ in $P_\Lambda$ if $x_1$ is incident with a downward spoke, and similarly we put $(x_k,\infty)$ (resp. $[x_k,\infty)$) in $P_\Lambda$ if $x_k$ is incident with a downward (resp. upward) spoke.

Define a quiver (that is, directed graph) $Q_\Lambda$ whose vertices are the elements of $P_\Lambda$ and with and edge joining $I$ to $J$ (in that orientation) if the closure of $J$ has nonempty intersection with $I$.  If there are $n$ spokes then this is a quiver of type $A_{n+1}$ (i.e. shaped like the Dynkin diagram $A_{n+1}$) whose edges are in natural bijection with the spokes of $\Lambda$: an upward spoke corresponds to a left-pointing arrow and a downward spoke to a right-pointing arrow.

\begin{theorem}
\label{thm:quiverquiver}
There is a natural equivalence of dg categories
$$\Sh(M;\Lambda) \cong \Rep(Q_\Lambda)$$
If $(x,\xi)$ belongs to a spoke of $\Lambda$ corresponding to an arrow $f$ of $Q_\Lambda$, then under this equivalence the functor $\mu_{x,\xi}$ intertwines with the functor ${\mathsf Cone}(f)$.
\end{theorem}
\begin{proof}
See \cite{STZ}.
\end{proof}
\begin{example}
\begin{enumerate}
\item Let $\raise3pt\hbox{$\bot$}\!\!\raise0.36pt\hbox{+}\!\!\hbox{\lower4.5pt\hbox{$\top$}}\subset T^*\bR$
be the union of the zero section, the fiber at $0,$ an upward spoke at some $x_- <0$ and a downward spoke at some $x_+>0.$  Then
$$
\Sh(\bR,\raise3pt\hbox{$\bot$}\!\!\raise0.36pt\hbox{+}\!\!\hbox{\lower4.5pt\hbox{$\top$}})\cong Rep( \bullet \leftarrow \bullet \leftarrow \bullet \rightarrow \bullet \rightarrow \bullet).
$$
\item Let $\Lambda = S^1 \cup T^*_{x_{0}}S^1 \hookrightarrow T^*S^1$ be the union of the zero section, and the cotangent fiber at some $x_0 \in S^1.$ Then
$$
\Sh(S^1, \Lambda) \cong Rep(\bullet \rightrightarrows \bullet).
$$
\end{enumerate}
\end{example}

We conclude this section, by showing how Theorem \ref{thm:quiverquiver} yields a very explicit description of the sheaf $\cpm(-)$. For concreteness, we focus on the example $X=S^1$ and $\Lambda = S^1 \cup T^*_{x_0}S^1$, the general case is similar. Denote $R^+$ and $R^-$ respectively the upward and downward spoke of $\Lambda$. We shall describe the sections of $\cpm(-)$ on \emph{contractible} open subsets $U \subset \Lambda$, and the assignment defining, on objects, the restriction functors 
$$
Res_U: \cpm(\Lambda) = Sh(X, \Lambda) \cong Rep(\bullet \rightrightarrows \bullet) \rightarrow \cpm(U).
$$ 
The definition on morphisms will be obvious. This is sufficient to reconstruct $\cpm(-)$. 
Let $M = (V_1 \underset{g}{\overset{f}{\rightrightarrows}}  V_2)$ be an object in $Rep(\bullet \rightrightarrows \bullet)$. Then,
\begin{itemize}
\item if $U \subset S^1$, $\cpm(U) \cong \bC\text{-mod}$, and $Res_U(M) = V_2$,
\item if $U \subset R^+$, $\cpm(U) \cong \bC\text{-mod}$, and $Res_U(M) = {\mathsf Cone}(f)$,
\item if $U \subset R^-$, $\cpm(U) \cong \bC\text{-mod}$, and $Res_U(M) = {\mathsf Cone}(g)$,
\item if $x_0 \in U$, $\cpm(U) \cong Rep(\bullet \leftarrow \bullet \rightarrow \bullet)$, and $Res_U(M)= (V_2 \stackrel{f}{\leftarrow} V_1 \stackrel{g}{\rightarrow} V_2$).
\end{itemize}

\subsection{Chordal ribbon graphs and $\cpm$}
\label{sec:chord}
Recall that a cyclic order $\cR$ on a set $S$ is a \emph{ternary} relation on $S$, which allows us to speak unambiguously about ordered triples, and satisifies the obvious properties enjoyed by a set of points arranged on an oriented circle. Thus, in particular, if $s \in S$, we can talk about the `successor' of $s$ with respect to $\cR$, which we denote $R(s)$. We call a pair of the form $(s, R(s)) \in S \times S$, a \emph{minimal pair}. 

We define a graph to be a pair $(D, V_D)$ where $D$ is a locally compact topological space, $V_X \subset X$ is a finite closed subset, and the open set $X - V_X$ is homeomorphic to a finite disjoint union of open intervals.  

\begin{definition}
Let $(D,V_D)$ be a graph in which every vertex has degree $\geq 2$.  A \emph{ribbon structure} on $(D,V_D)$ is a collection $\{\cR_v\}_{v \in V_D}$ where $\cR_v$ is a cyclic order on the set of half-edges incident with $v$.  We call a graph equipped with a ribbon structure a ribbon graph.
\end{definition}

\begin{definition}
\label{def:chordal}
A \emph{chordal ribbon graph} is a pair $(X,Z)$, where
\begin{itemize}
\item $X$ is a ribbon graph with vertices of valency at most $4$,
\item $Z$ is a closed subgraph, with vertices of valency at most $2$, containing each vertex of $X$.
\end{itemize}
Also, we require that if $v \in V_X$ is a $4$-valent vertex, then $v$ has valency $2$ in $Z$, and the two half-edges belonging to $Z$ do not form a minimal pair in the cyclic order $\cR_v$.
We refer to $Z$ as the \emph{zero section} of the chordal ribbon graph.
\end{definition}

Let $\Chord$ denote the category whose objects are chordal ribbon graphs, and where $\Hom((C,W),(D,Z))$ is given by the set of open immersions $j:C \hookrightarrow D$ which preserve the cyclic orders at each vertex, and such that $j(W) \subset Z$. We endow $\Chord$ with a Grothendieck topology in the evident way.

The simplest examples of chordal ribbon graphs, called \emph{fishbones}, are pairs of the form $(\Lambda, X \cap \Lambda)$, where $X$ is a $1$-dimensional manifold, and $\Lambda \subset T^*X$ is a closed conical Lagrangian subset. In fact, $\Lambda$ is equipped with a canonical ribbon graph structure, while the subgraph $X \cap \Lambda$ clearly has all the properties of a \emph{zero section}, in the sense specified by Definition \ref{def:chordal}. Section \ref{sec:mic} gives a recipe for constructing a sheaf $\cpm(-)$ on any fishbone $(\Lambda, X \cap \Lambda)$ (see Definition \ref{def:cpm1}). As the full subcategory of fishbones is a basis for the Grothendieck topology on $\Chord$,\footnote{In fact, making this assertion precise requires defining morphisms in $\Chord$ in a more careful manner than we did above, we refer the reader to \cite{STZ} for further details.} we can make the following definition.

\begin{definition}
Denote $\cpm:\Chord \to \dgCat$ the sheaf of dg categories on $\Chord$ whose restriction to the sub-category of fishbones recovers Definition \ref{def:cpm1}. We call $\CPM(D,Z)$ the \emph{constructible plumbing model} of the chordal
ribbon graph $(D,Z)$.
\end{definition}

Chordal structure and restriction on valency are just convenient technical assumptions which could be removed as $\cpm(-)$ is expected not to depend on them. More precisely, up to quasi-equivalence, the constructible plumbing model of the chordal ribbon graph $(D, Z)$ should be a function solely of the `deformation class,' appropriately defined, of the ribbon graph $D$.\footnote{\label{foot:def} In fact, any ribbon graph is deformation equivalent, in this sense, to a ribbon graph admitting chordal structure. A more satisfactory definition of $\cpm$, which will take as input appropriately graded ribbon graphs of abitrary valency, is currently work in progress.}

Setting technical complications aside, let us assume for the moment that $\cpm(-)$ can be evaluated on a general ribbon graph. Then the expected relationship with the Fukaya category can be formulated as in Conjecture \ref{conj} below.
Recall that ribbon graphs label cells in the moduli space of punctured Riemann surface (see e.g. \cite{H} and \cite{P}). Further, if $\Sigma$ lies in the cell labeled by $\Gamma_{\Sigma}$, there is an embedding $\Gamma_{\Sigma} \hookrightarrow \Sigma$, and a nicely behaved retraction of $\Sigma$ onto $\Gamma_{\Sigma}$. In the language of Stein geometry, $\Gamma_{\Sigma}$ is the \emph{skeleton} of $\Sigma$.
\begin{conjecture}
\label{conj}
Let $\Sigma$ be a punctured Riemann surface with skeleton $\Gamma_{\Sigma}$, then $\cpm(\Gamma_{\Sigma})$ is quasi-equivalent to $Fuk(\Sigma)$.\footnote{Here and elsewhere in this paper, when referring to the Fukaya category, we actually mean the \emph{split closure} of its category of \emph{twisted complexes}, see \cite{S}.}
\end{conjecture} 

\section{Homological Mirror Symmetry for tcnc curves}
\label{sec:hms}
In this section we will prove the main theorem of \cite{STZ}, which establishes a version of homological mirror symmetry for a class of nodal, stacky, curves of genus $g \leq 1$, introduced in Section \ref{sec:tcnc} below. The proof of HMS will be discussed in Section \ref{sec:proof}, and will make use of the model for the Fukaya category supplied by the sheaf $\cpm(-)$.
\subsection{tcnc curves}
\label{sec:tcnc}
Let $\bP^1(a_1,a_2)$ be a projective line, with stacky points at $0$, and $\infty$, and isotropy groups isomorphic, respectively, to $\bZ_{a_1}$, and $\bZ_{a_2}$.\footnote{Note that our conventions differ from the ones commonly found in the literature. Weighted projective lines, which are denoted $\bP^1(a_1,a_2)$, are ususally defined as quotients of $\bC^2 - \{0\}$ by $\bC^*$ acting with weights $a_1, a_2$. According to the latter definition, if $gcd(a_1,a_2) \neq 1$, $\bP^1(a_1,a_2)$ has non-trivial generic isotropy group. However, the two definitions agree if $gcd(a_1,a_2) = 1$.} We call $\bP^1(a_1,a_2)$ a \emph{Beilinson-Bondal} (or, BB) curve. 

\begin{definition} 
\label{def:tcnc}
A \emph{tcnc curve} $C$ is a connected, reduced DM stack of dimension $1$, with nodal singularities, such that its normalization $\tilde{C} \stackrel{\pi}{\rightarrow} C$ is a disjoint union of $n$ BB curves $P_1, \dots, P_n$. Further, if $Z \hookrightarrow C$ is the singular set, we require that $\pi^{-1}(Z)$ interesects each $P_i$ in at most two points.
\end{definition}

It follows from the definition, that the coarse moduli space of a tcnc curve must have arithmetic genus $g \leq 1$, and thus be equal to a cycle of rational curves (i.e., a Galois cover of a nodal $\bP^1$), if $g=1$, and to a chain of rational curves if $g=0$. 

A tcnc curve $C$ is uniquely determined by its genus, together with a tuple of positive integers, which we shall call the $W$-vector, and which specifies the orders of the isotropy groups at points $0$ and $\infty$, on the different irreducible components of $C$. We will not give a formal definition of the $W$-vector, as it easier to see how this works in an example.

\begin{example}
\label{ex:tcnc}
Consider the weighted projective plane $\bP^2(1,2,3) = [(\bC^3 - \{0\})/\bC^*]$, where $\bC^*$ acts with weights $1,2,3$.  
\begin{itemize}
\item Let $C \hookrightarrow \bP^2(1,2,3)$ be the sub-stack defined by the equation $x_0x_1=0$. $C$ is a tcnc curve of genus $0$, and can be encoded in the $W$-vector $(1,2,3) \in \bN^3$. Note that the reverse tuple $(3,2,1)$ is an equally valid $W$-vector for $C$.
\item Let $C' \hookrightarrow \bP^2(1,2,3)$ be defined by $x_0x_1x_2 = 0$. $C'$ has genus $1$, and is also described by the $W$-vector $(1,2,3) \in \bN^3$. As before, because of the evident symmetries of $C'$, there are other viable choices of $W$-vector for $C'$, such as for instance $(2,3,1)$.
\end{itemize} 
\end{example}

\begin{figure}
\includegraphics[height=2in]{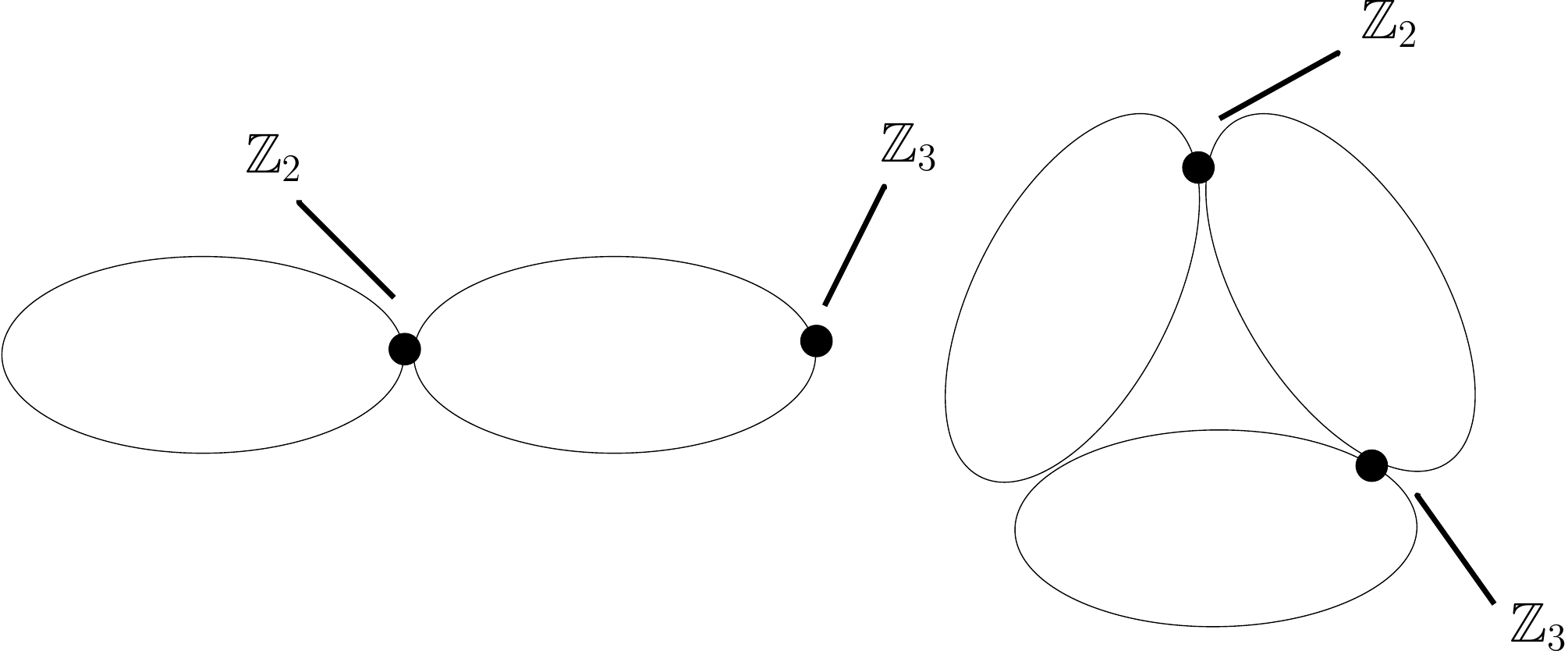}
\caption{Above is a picture of the tcnc curves considered in Example \ref{ex:tcnc}. The labels indicate the isotropy subgroups at the stacky points.}
\end{figure}

\begin{definition}
Denote $C_A^i$ the tcnc curve of genus $i \in \{0,1\}$, with $W$-vector $A \in \bN_{>0}^m$. 
\end{definition}
Theorem \ref{thm:perf} gives a description of the category of perfect complexes over a tcnc curve which will play a key role in our proof of homological mirror symmetry.
\begin{theorem}
\label{thm:perf}
Let $C$ be a tcnc curve with singular set $Z$,\footnote{Note that $Z$ is a disjoint union of classifying stacks of the form $[Spec(\bC)/\mu_{a_i}]$.} and normalization $\pi:\tilde{C} \rightarrow C$. Let $\sigma,\tau$ be two non overlapping sections of $\pi^{-1}(Z) \rightarrow Z$, then the diagram 
$$
\xymatrix{
\ \Perf(C) \ar[r]^{\pi^{*}}   & \Perf(\tilde{C}) \ar@<1ex>[r]^{\sigma^*} \ar@<-0.5ex>[r]_{\tau^*}& \Perf(Z)}
$$
is an equalizer in $\dgCat$.  
\end{theorem} 
\begin{proof}
The statement, for nodal curves of arbitrary genus, is essentially proved in Section 4 of \cite{L}. The small adjustments required to handle nodal DM stacks of dimension $1$, and thus extend the argument to tcnc curves, are worked out in \cite{STZ}. 
\end{proof}

\subsection{Wheels, dualizable ribbon graphs, and HMS}
\label{sec:proof}
A \emph{wheel} is a conical Lagrangian $\Lambda$ in $T^* S^1$ that contains the zero section. We can equip a wheel with canonical chordal structure, given by the pair $(\Lambda, Z=S^1)$. Recall from Section \ref{sec:chord} that a choice of orientation on $S^1$ yields a subdivision of the spokes of $\Lambda$ into two groups, called respectively ``upward'' and ``downward.'' We will denote $\Lambda_{a_1,a_2}$ a wheel with $a_1$ upward spokes, and $a_2$ downward spokes. 

\begin{theorem}
\label{thm:wheels}
If $a_1, a_2 \in \bN_{>0}$, there is an equivalence $\Perf(\bP^1(a_1,a_2)) \cong Sh(S^1, \Lambda_{a_1,a_2})$.
\end{theorem}
\begin{proof} 
Theorem \ref{thm:wheels} is due to Bondal \cite{B}, who first suggested this should be interpreted as an instance of mirror symmetry. Partially inspired by Bondal's insights, Fang, Liu, Treumann and Zalow develop an approach to HMS for (stacky) toric varieties \cite{FLTZ1, FLTZ2}, which in particular implies this result, and is the starting point for the project pursued in \cite{STZ}. Note that
when $a_1=a_2=1$, this recovers the classic result of Beilinson \cite{Be}, according to which there is an equivalence 
$
D^b(Coh(\bP^1)) \cong Rep(\bullet \rightrightarrows \bullet).
$
In fact, Theorem \ref{thm:quiverquiver} gives an equivalence $Sh(S^1, \Lambda_{1,1}) \cong Rep(\bullet \rightrightarrows \bullet)$. 
\end{proof}

\begin{remark}
\label{rem:stalk}
Theorem \ref{thm:wheels} can be refined, by requiring that the equivalence intertwine appropriate `stalk functors.'
Let $i \in \{1,2\}$, and consider the inclusion $j_i: [*/\bZ_{a_i}] \rightarrow \bP^1(a_1,a_2)$. If $\chi$ is a character of $\bZ_{a_i}$, we denote $S^i_{\chi}$ the following composition
$$
S^i_{\chi}: \Perf(\bP^1(a_1,a_2))  \stackrel{j_i^* }{\longrightarrow} \Perf([*/\bZ_{a_i}]) \stackrel{{\chi} }{\longrightarrow} \bC\text{-mod}.
$$
Choose a labelling of the downward spokes of $\Lambda_{a_1,a_2}$ by characters of $\bZ_{a_1}$, and of the upward spokes, by characters of $\bZ_{a_2}$.\footnote{Both the set of characters and the set of up-/down- ward spokes come with natural cyclic orders (the spokes inherit it from the ribbon structure on $\Lambda_{a_1,a_2}$). The labelling cannot therefore be arbitrary, as it must preserve this cyclic order, see \cite{STZ}.} Denote $R^i_{\chi}$ the spoke of $\Lambda_{a_1,a_2}$ labeled by the character $\chi$ of $\bZ_{a_i}$. Note that there is a restriction functor
$$ 
Res^i_{\chi}: \cpm(\Lambda_{a_1,a_2}) \longrightarrow \cpm(R^i_{\chi}) \cong \bC\text{-mod}.
$$
The claim is that we can define $\Phi: \Perf(\bP^1(a_1,a_2)) \cong \cpm(\Lambda_{a_1,a_2})$ in such a way that we get commutative diagrams of dg categories
$$
\xymatrix{
 \Perf(\bP^1(a_1,a_2)) \ar[r]^{S^i_{\chi}} \ar[d]_{\Phi} & \bC\text{-mod} \ar[d]^{\cong}\\
 \cpm(\Lambda_{a_1,a_2}) \ar[r]^{Res^i_{\chi}} & \cpm(R^i_{\chi}). 
}
$$ 
\end{remark}

The chordal ribbon graphs which are most relevant in the context of mirror symmetry are of a special kind, called \emph{dualizable}. Dualizable ribbon graphs are obtained by gluing together wheels along matching sets of up- and down-ward spokes. We will limit ourselves to explain the geometry of dualizable ribbon graphs through concrete examples, while referring the reader to \cite{STZ} for rigorous definitions. Also, we will mostly consider \emph{trivalent} dualizable ribbon graphs, as this will somewhat simplify the exposition, and will not reduce generality in any serious way (in fact, any chordal ribbon graph is, in an appropriate sense, `deformation equivalent' to a trivalent graph, cf. Footnote \ref{foot:def}).  

Let $a \in \bN_{>0}$, and denote $R_a$ the chordal ribbon graph given by a disjoint union of positive rays, with empty vertex set, and trivial chordal structure, $R_a = (\coprod_{1 \leq i \leq a}\bR_{>0}, \varnothing)$. If $\Lambda_{a_1, a_2}$ is a wheel, we can choose morphisms in $\Chord$
$$R_{a_2} \stackrel{i^-}{\rightarrow} \Lambda_{a_1, a_2} \stackrel{i^+}{\leftarrow}R_{a_1},$$
mapping homeomorphically the components of $R_{a_1}$, and $R_{a_2}$, respectively onto the upward, and downward, spokes of $\Lambda_{a_1,a_2}$. 
  
\begin{example}
\label{ex:dual}
\begin{enumerate}
\item Let $A = (1,2,3) \in \bN^3$, and denote $\Lambda^0_A = (D^0_A, Z_A)$ the chordal ribbon graph obtained as the push-out of the following diagram in $\Chord$,
$$
\xymatrix{
& R_2 \ar[ld]^{i^+} \ar[rd]_{i^-} \\
\Lambda_{1,2} & & \Lambda_{2,3}
}
$$ 
That is, $D^0_A$ is the push-out of the underlying $1$-dimensional CW-complexes, and is equipped with the unique chordal structure rendering the natural inclusions $$\Lambda_{1,2} \hookrightarrow \Lambda^0_A \hookleftarrow \Lambda_{2,3}$$ morphisms in $\Chord$. Thus, $Z_A$ is the disjoint union of two circles. Note that $D^0_A$ is the non-compact skeleton of a punctured curve of genus $0$, endowed with appropriate Stein structure.
\item Let $A = (1,2,3) \in \bN^3$, and let $\Lambda^1_A = (D^1_A, Z_A)$ be the push-out of the following diagram in $\Chord$
$$
\xymatrix{
& R_2 \ar[ld]^{i^+} \ar[rd]_{i^-} & & R_3 \ar[ld]^{i^+} \ar[rd]_{i^-} & _{i^-} & R_1 \ar[dl]^{i^+} \ar[dlllll] \\
\Lambda_{1,2} & & \Lambda_{2,3} & & \Lambda_{3,1} 
}
$$ 
The ribbon graph $D^1_A$ is isomorphic to the skeleton of a Stein torus with $6$ punctures.
\end{enumerate}
\end{example}  

\begin{figure}
\includegraphics[height=2in]{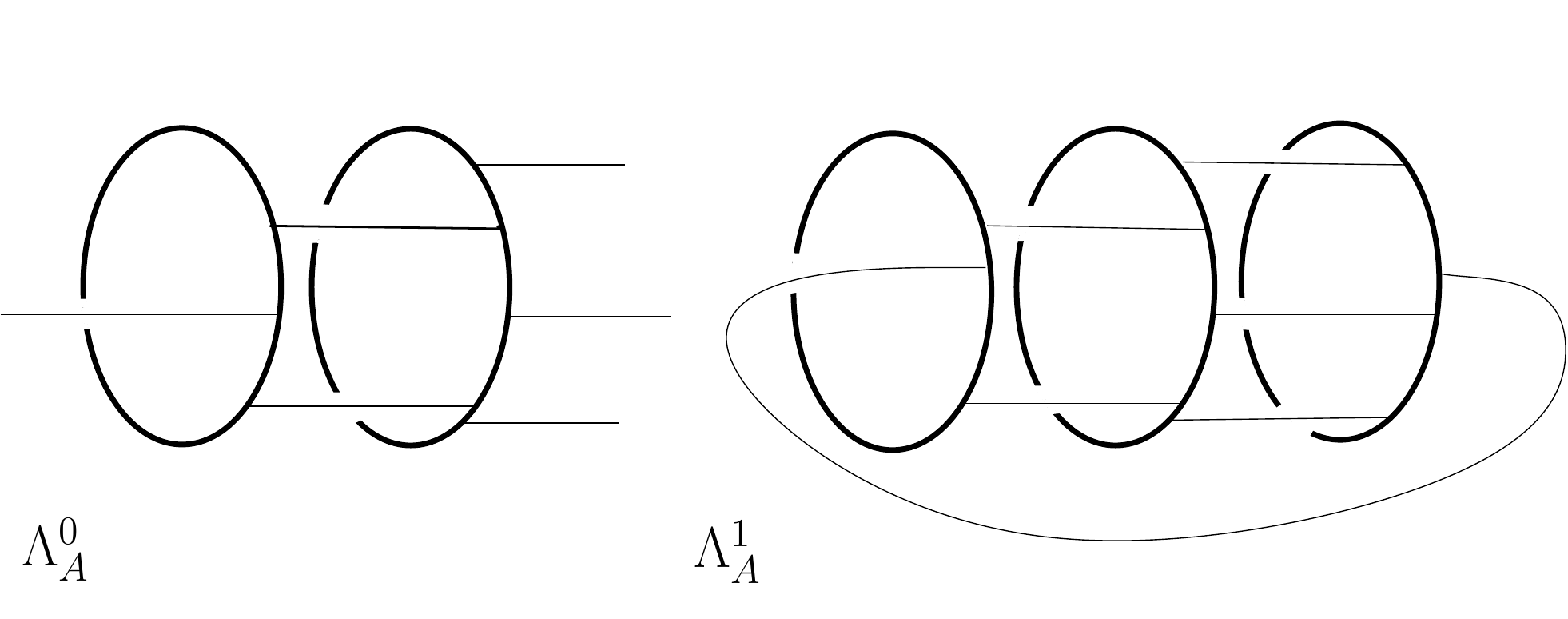}
\caption{The dualizable ribbon graphs considered in Example \ref{ex:dual} (1) and (2) are sketched above. We have signaled the zero section by drawing it with a thicker line.}
\label{fig:dual}
\end{figure}

Dualizable ribbon graphs are constructed by adjoining together wheels as in the two examples above,\footnote{\label{foot:braid} It is important to point out that, as shown in Figure \ref{fig:dual}, in a dualizable ribbon graph the strands joining together the components of the zero section cannot be (non-trivially) `braided.' This can be translated in appropriate conditions of coherency on the maps $R_{a_i} \rightarrow \Lambda_i$. We refer the reader to \cite{STZ} for further details.} and are therefore completely determined up to isomorphism by their genus, which is equal to $0$ or $1$,\footnote{The genus of a ribbon graph $D$ can be described geometrically as the genus of any surface in which $D$ can be embedded, in a way compatible with the ribbon structure, as a deformation retract. Thus $D^0_A$ in example \ref{ex:dual} (1) has genus $0$, while $D^1_A$ in Example \ref{ex:dual} (2), has genus $1$. For a formal, combinatorial definition of the genus of a ribbon graph, see \cite{STZ}.} and by a tuple of positive integers recording the number of edges connecting the different components of the chordal zero section, $Z$. This is entirely analogus to the case of tcnc curves, which was discussed in Section \ref{sec:tcnc}. 

Let $i \in \{0,1\}$, and let $A = (a_1, \dots, a_m)$ be a tuple of positive integers, and denote $\Lambda^i_A$ the dualizable ribbon graph identified, in the manner explained above, by the pair of $i$ and $A$.

\begin{theorem}[HMS]
\label{thm:hms}
There is an equivalence of dg categories
$$
\Perf(C^i_A) \cong \cpm(\Lambda^i_A).
$$
\end{theorem}

\begin{proof}
There is a covering of $\Lambda^i_A$ given by wheels $W_i= \Lambda_{a_i, a_{i+1}}$. Then, by the sheaf property of $\cpm$ we have an equalizer diagram
$$\CPM(C^0_A) \to \CPM(\coprod W_i) \rightrightarrows \CPM(\coprod W_i \cap W_{i+1}).$$
The Theorem then follows immediately from Theorem \ref{thm:wheels} (and Remark \ref{rem:stalk}), and Theorem \ref{thm:perf}. 
\end{proof}

As discussed above, dualizable ribbon graphs $\Lambda^i_A$ arise as skeleta of punctured curves of genus $i$ with appropriate Stein structure. Granting Conjecture \ref{conj}, Theorem \ref{thm:hms} can therefore be interpreted as a HMS statement, relating punctured symplectic surfaces, and degenerate, nodal algebraic curves, having equal genus $i \in \{0,1\}$. In particular, this confirms the well known mirror symmetry heuristics according to which the mirror of a symplectic torus with $n$ punctures should be a a cycle of $n$ rational curves.\footnote{Kontsevich announced related results in \cite{K1}. HMS for the nodal $\bP^1$ is also treated in \cite{LPe}.}

\section{Tcnc curves and Fourier-Mukai equivalences}
\label{sec:deq}
The Fukaya category of a punctured Riemann surface $\Sigma$ should depend solely on the symplectic geometry of $\Sigma$, which is encoded in its genus, and in the number of punctures. In view of Conjecture \ref{conj}, this suggests that if $D$ and $D'$ are (chordal) ribbon graphs arising as skeleta of a unique punctured surface $\Sigma$ equipped with two different Stein structures, there should be an equivalence $\cpm(D) \cong \cpm(D')$. 

In this section we sketch a proof that this is indeed the case for dualizable ribbon graphs, by introducing a simple graphical calculus which will enable us to construct this equivalence in a step-by-step fashion. A precise statement of our theorem is collected below. If $n \in \bN$, we denote ${\bf 1}(n) \in \bN^n$ the $n$-tuple filled with $1$-s. By slight abuse of notation, we shall also denote $(a, {\bf 1}(n), b)$ a tuple of length $2 + n$, of the form $(a, 1, 1, \dots, 1, b)$.

\begin{theorem}
\label{thm:move}
If $A = (a_1, \dots, a_m)$ is a tuple of positive integers, there are equivalences
\begin{enumerate}
\item $\cpm(\Lambda^0_A) \cong \cpm(\Lambda^0_{A'})$, where $A'= (a_1, {\bf 1}(a_2 + \dots + a_{m-1}), a_m)$,
\item $\cpm(\Lambda^1_A) \cong \cpm(\Lambda^1_{A'})$, where $A' = {\bf 1}(a_1 + \dots + a_n)$.
\end{enumerate}
\end{theorem}

Our interest in this result depends on the fact that, using the dictionary provided by Theorem \ref{thm:hms}, it can be translated in a statement regarding derived equivalences of tcnc curves.
\begin{corollary}
\label{cor:fm}
If $A = (a_1, \dots, a_m)$ is a tuple of positive integers, there are equivalences
\begin{enumerate}
\item $\Perf(C^0_A) \cong \Perf(C^0_{A'})$, where $A'= (a_1, {\bf 1}(a_2 + \dots + a_{m-1}), a_m)$,
\item $\Perf(C^1_A) \cong \Perf(C^1_{A'})$, where $A' = {\bf 1}(a_1 + \dots + a_n)$.
\end{enumerate}
\end{corollary}

Denote $X_n$ a cycle of rational curves with $n$ components. Corollary \ref{cor:fm} implies in particular that there is an equivalence $\Perf(X_n) \cong \Perf([X_1/\mu_n])$, where $\mu_n$ is the group of $n$-th roots of unity, acting on $X_1$ in the obvious manner, and $[X_1/\mu_n]$ is the quotient stack. As we shall explain, this result can be interpreted as a generalization to the singular case of Mukai's classic work on derived equivalences of smooth elliptic curves (and, more generally, of principally polarised abelian varieties) \cite{M}. Recall that Mukai shows that, if $X$ and $X^{\vee}$ are dual abelian varieties, there is a non-trivial equivalence $D^b(Coh(X)) \cong D^b(Coh(X^{\vee}))$, which he defines via a pull-push formalism, by taking as kernel the universal bundle on the product $X \times X^{\vee}$. 

As in the smooth case, the nodal projective line $X_1$ is isomorphic to its dual $X_1^{\vee}$, which is the moduli space of rank $1$, degree $0$, torsion-free sheaves over $X_1$ \cite{BK1}. Further, one can show that $X_n$, which is the $n$-fold cover of $X_1$, parametrizes \emph{$\mu_n$-equivariant} sheaves on $X_1$ satisfying the properties just listed. In this perspective, we can interpret the covering map $X_n \rightarrow X_1$ as induced by `forgetting the equivariant structure.' Thus, $X_n$ is isomorphic to the moduli space of rank $1$, degree $0$, torsion-free sheaves over the quotient stack $[X_1/\mu_n]$ or, in other words, $X_n$ is dual, in the sense discussed above, to $[X_1/\mu_n]$.

The existence of an equivalence $\Perf(X_n) \cong \Perf([X_1/\mu_n])$ therefore fits well with what we would expect based on the smooth case.\footnote{Note that any such equivalence would extend to an equivalence of the full derived categories, see Theorem 1.2 in \cite{Ba1}.} Note that the case $n=1$ was also studied by Burban and Kreussler \cite{BK1}, who use the theory of spherical functors to define a non-trivial derived equivalence $D^b(Coh(X_1)) \cong D^b(Coh(X_1))$ having the required properties.\footnote{In \cite{Si}, extending results of \cite{BK1}, we defined an action of the mapping class group of a torus with $n$ punctures on $D^b(Coh(X_n))$. The argument we shall describe below can be interpreted, roughly, as defining an action of an appropriate version of the mapping class \emph{groupoid}. For a definition of spherical functor, see \cite{ST}.} 

\subsection{Elementary moves}
\label{sec:em}
In this section we introduce a set of operations, called \emph{elementary moves}, which can be used to modify the geometry of chordal ribbon graphs while preserving the global sections of $\cpm(-)$. First, however, we spell out the behaviour of $\cpm(-)$ on some especially simple chordal ribbon graphs, which can be used as building blocks for all trivalent graphs in $\Chord$.

Let $E$ be a ribbon graph with empty vertex set, and underlying $CW$ complex homeomorphic to $\bR$. We can equip $E$ with two distinct chordal structures $(E, W)$, by setting either $W = E$, or $W = \varnothing$. In both cases, $\cpm(E,W) \cong \Cmod$. Thus, if $(D,Z)$ is a chordal ribbon graph, and $\{e_i\}_{i \in I}$ is the set of edges of $D$, restriction to the edges yields \emph{stalk} functors, indexed by $I$, 
$ 
Res_i:\cpm(D,Z) \rightarrow \Cmod \cong \cpm(e_i),
$
which generalize the `microlocal stalks' discussed in Section \ref{sec:mic}. It is often convenient to indicate an object $L \in \cpm(D,Z)$ by assigning the collection of its stalks $Res_i(L)$, which can be visualized as labels attached to the edges $e_i$ of $D$. 

A \emph{pitchfork} is a chordal ribbon graph $P = (D, Z)$, such that $D$ is isomorphic to the union of the real line $\bR$, and an upward spoke $R^+$. As shown by Figure \ref{fig:pitch} above, there are only two possible choices of zero section, which yield inequivalent chordal ribbon graphs $P_1$, $P_2$. In either case, using Theorem \ref{thm:quiverquiver}, we can see that the global sections of $\cpm(-)$ are given by $Rep(\bullet \rightarrow \bullet)$. The edges of the graphs represented in Figure \ref{fig:pitch} are decorated with labels corresponding to an object $L = (V \stackrel{f}{\rightarrow} W) \in Rep(\bullet \rightarrow \bullet)$. Thus, for instance, the picture indicates that the stalk of $L \in \cpm(P_1) = Rep(\bullet \rightarrow \bullet)$ on any point lying on the edge $e_2$, is isomorphic to $\Cone(f)$. 

\begin{figure}

\includegraphics[height=2in]{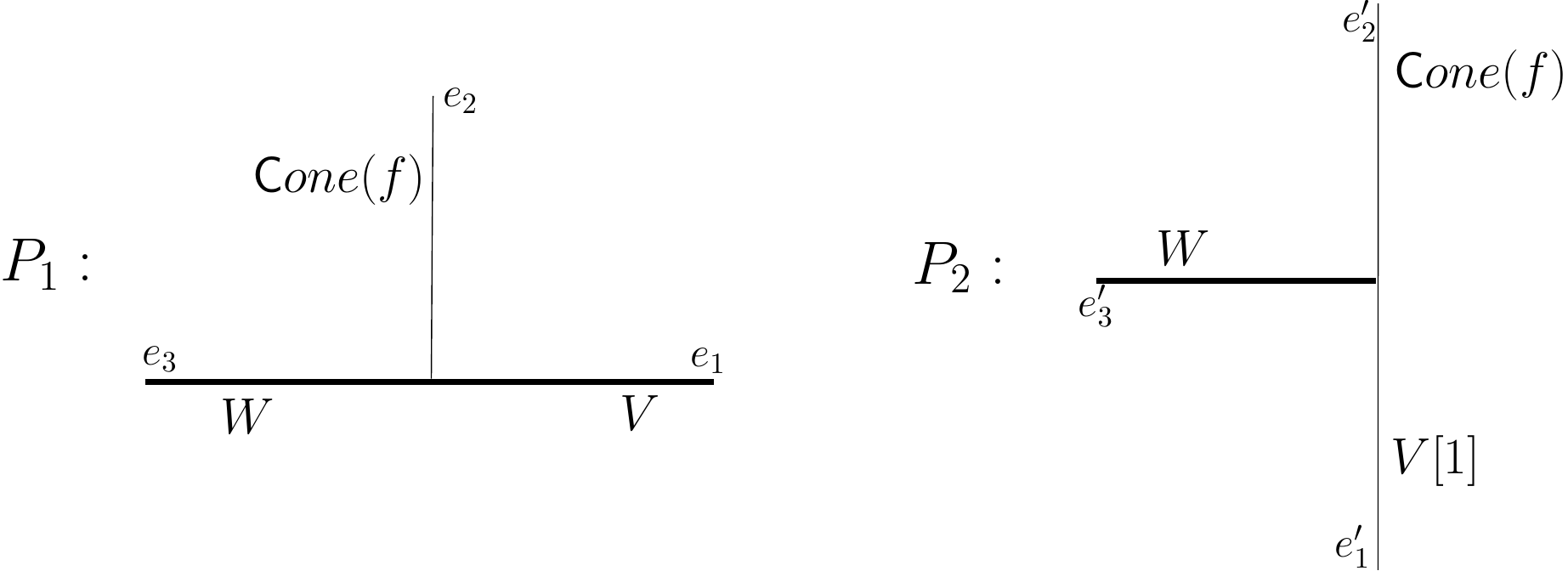}
\caption{Up to isomorphism, there are only two chordal structures on a pitchfork, which are represented above, and are denoted $P_1$ and $P_2$.}
\label{fig:pitch}
\end{figure}

The set of the \emph{Elementary Moves}, or $EM$-s, which we shall use in the proof of Theorem \ref{thm:move}, is given in the table below (Figure \ref{fig:moves}). Note that the ribbon graphs considered in Figure \ref{fig:moves} are obtained by gluing together pitchforks along common edges, and thus we can easily compute the sections of $\cpm(-)$ over them using Lemma \ref{lem:ho eq}.

\begin{figure}
\includegraphics[height=6in]{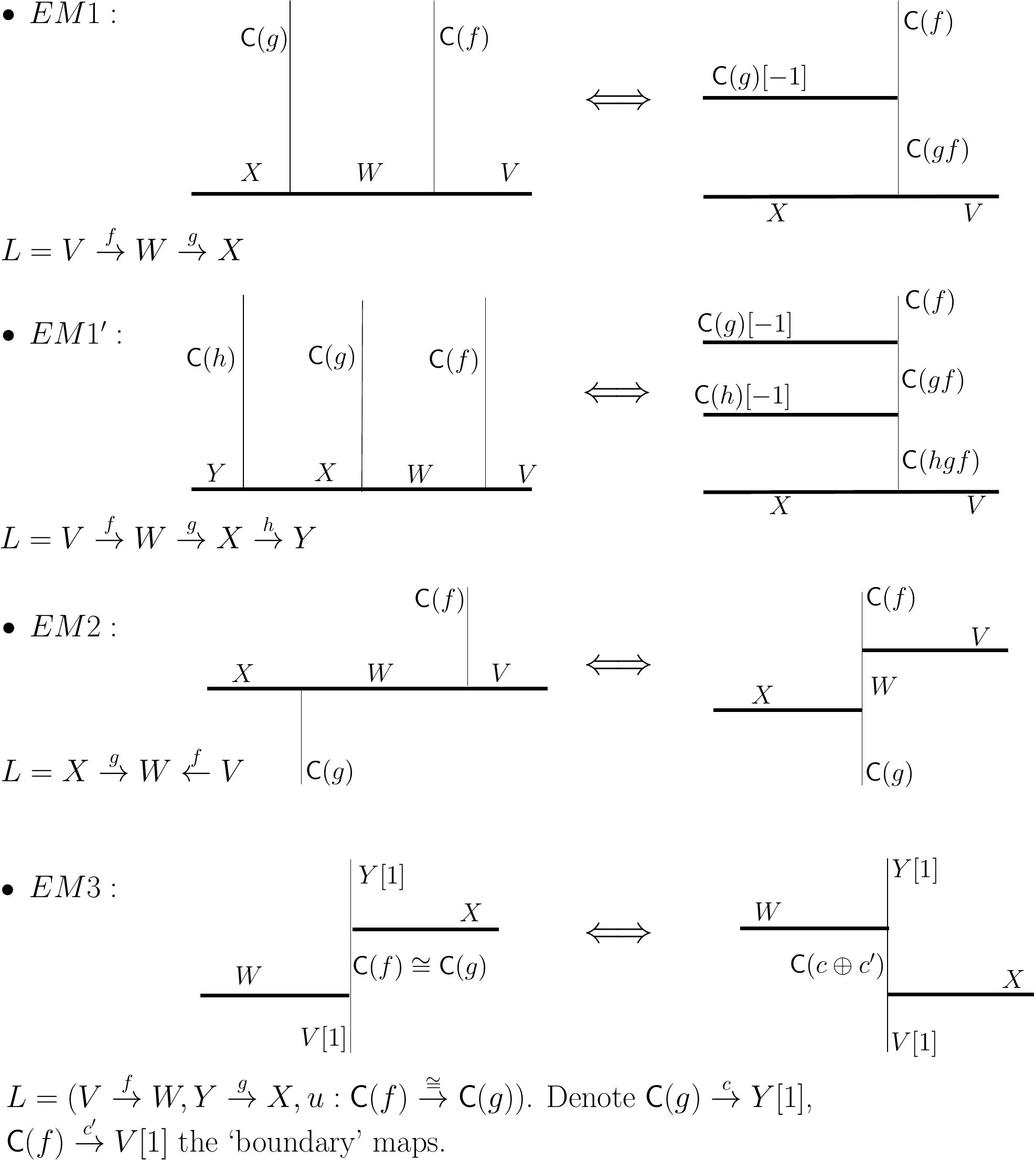}
\caption{}
\label{fig:moves}
\end{figure}

For each elementary move $EM_i$, let $(D_{l_i}, Z_{l_i})$ be the graph appearing on the left of the `$\Leftrightarrow$' symbol, and $(D_{r_i}, Z_{r_i})$ the graph appearing on the right. $EM$-s preserve global sections of $\cpm(-)$, and there is a preferred isomorphism $\Phi_i: \cpm(D_{l_i}, Z_{l_i}) \cong \cpm(D_{r_i}, Z_{r_i})$. We have labelled the edges of $(D_{l_i}, Z_{l_i})$, and $(D_{r_i}, Z_{r_i})$, with the stalks of $L \in \cpm(D_{l_i}, Z_{l_i})$, and $\Phi_i(L) \in \cpm(D_{r_i}, Z_{r_i})$, respectively. This schematics gives enough information for reconstructing $\Phi_i$ entirely. Below, we shall see how this works in a concrete example. Also, in order to simplify notations, we have indicated the cone of a map $f: V \rightarrow W \in \Cmod$ simply by $\mathsf{C}(f)$. 

Although $EM1'$ can be obtained simply by iterating $EM1$, we have inserted it in the table given in Figure \ref{fig:moves}, because in Section \ref{sec:pfmo} it will be convenient to apply this transformation directly, without factoring it into simpler $EM$-s. 
The proof that elementary transformations do in fact preserve sections of $\cpm(-)$ is not hard, and we omit the details. However, as an example of the kind of arguments involved, it might be useful to discuss briefly the case of $EM1$.

Note that, following the notations of Figure \ref{fig:pitch},  $(D_{r_1}, Z_{r_1})$ can be constructed by gluing together edges $e_2$ of $P_1$, and $e'_3$ of $P_2$. An object in $\cpm(D_{r_1}, Z_{r_1})$ is given therefore by a triple of the form
$
(A \stackrel{m}{\rightarrow} B, C \stackrel{n}{\rightarrow} D, u: \Cone(m) \stackrel{\cong}{\rightarrow} \Cone(n)).    
$
Now, take $L = (V\stackrel{f}{\rightarrow}W\stackrel{g}{\rightarrow}X)$ in $\cpm(D_{l_1}, Z_{l_1})$. Since $\cpm(D_{l_1}, Z_{l_1})$ is a dg (pre-)triangulated category, it follows from the \emph{octahedral axiom} that there is a map 
$p: \mathsf{C}(gf)[-1] \rightarrow \mathsf{C}(g)[-1]$, and a natural isomorphism $\mathsf{C}(p) \cong \mathsf{C}(f)$. Thus we can define $\Phi_1$ on objects by setting
$$
\Phi_{1}(L) = (V \stackrel{f}{\rightarrow}W, \mathsf{C}(gf)[-1] \stackrel{p}{\rightarrow} \mathsf{C}(g)[-1], \mathsf{C}(f) \cong \mathsf{C}(p)),
$$
and the definition on morphisms is obvious.

\subsection{The proof of Theorem \ref{thm:move}}
\label{sec:pfmo}
Figure \ref{fig:blocks} represents two different kinds of chordal ribbon graphs, which are denoted $A_n$, and $B_n$, with $n \in \bN_{>0}$. All tri-valent dualizable ribbon graph can be assembled by gluing along their external edges a certain number of copies of graphs of type $A$ and $B$. In order to prove Theorem \ref{thm:move}, it is therefore enough to show that for every $n$ there exists an equivalence $\Phi_n:\cpm(A_n) \cong \cpm(B_n)$, with the property that $\Phi_n$ preserves the stalks on the $4$ external edges $e_i$, and $e'_i$.\footnote{It might be surprising that the strands of $A_n$ are `non-trivially braided.' The existence of the equivalence $\Phi_n$ depends, in fact, in a crucial way on the choice of this particular geometry. Note however that the zero section of a dualizable ribbon graph is a union of loops. Considered as edges of the larger graph, the strands in a subgraph of type $A$ can therefore be un-braided, cf. also Footnote \ref{foot:braid}.} In fact, starting with any dualizable ribbon graph, we can turn it into a dualizable graph having weight vector with entries all equal to $1$ by successively replacing its subgraphs of type $A$ with subgraphs of type $B$. The availabilty of the equivalences $\Phi_n$ insures that, while doing so, we are not affecting the sections of $\cpm(-)$ (up to isomorphism).

\begin{figure}[ht]
\includegraphics[height=2in]{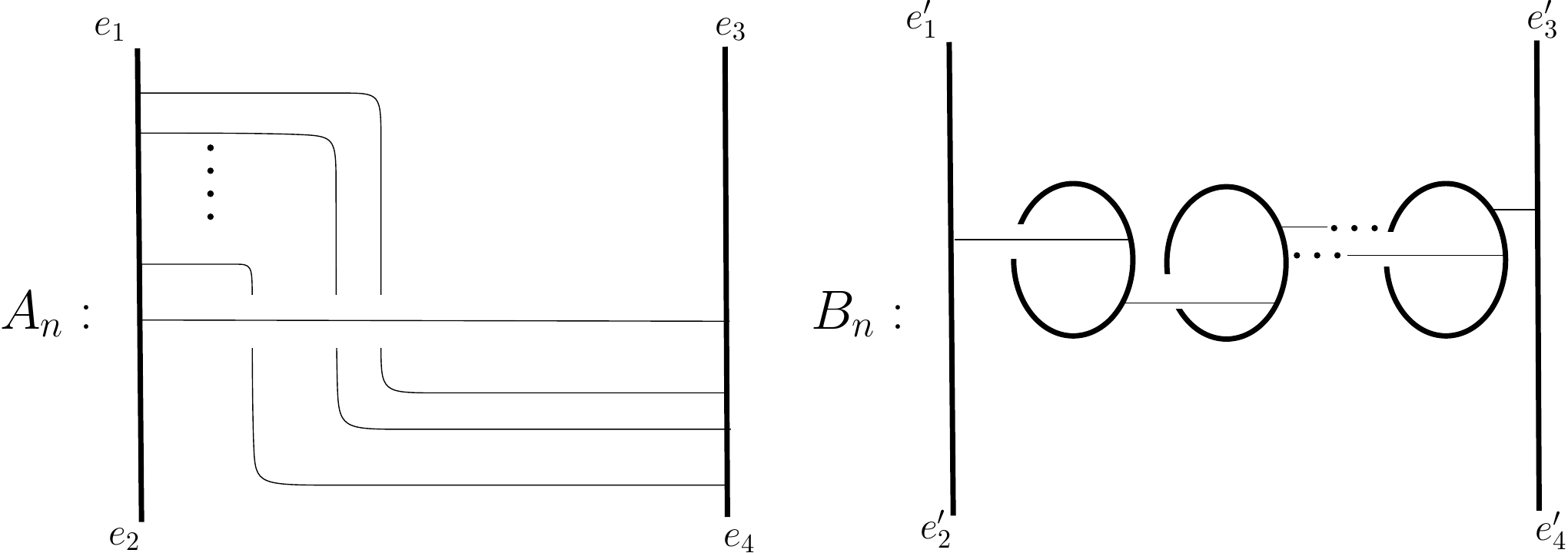}
\caption{The parameter $n \in \bN_{>0}$ indicates the number of strands in $A_n$, and loops in $B_n$.}
\label{fig:blocks}
\end{figure}
 
To define the $\Phi_n$-s, we have to break down the algorithm just described in yet smaller subroutines. For all $n \in \bN_{>0}$, we identify suitable subgraphs of $A_n$ isomorphic to the graphs appearing in Figure \ref{fig:moves}, and we modify their geometry via the appropriate elementary move. This gives rise to a new graph, which we can manipulate in similar manner, until, after a finite number of steps, we achieve the geometry of $B_n$. 
This procedure involves keeping track of what happens to the stalks across $EM$-s, to make sure that our operations, which are local in nature, determine equivalences at the level of global sections of $\cpm(-)$. This can be easily done, using the information on stalks given by Figure \ref{fig:moves}.

The proof of Theorem \ref{thm:move} can therefore be reduced to a simple graphical calculus, which is illustrated in Figures \ref{fig:n2}, and \ref{fig:n3} below, for the cases $n=2$, and $n=3$. At each step we apply an elementary move, which is explicitly indicated over the symbol `$\Leftrightarrow$.' It is important to notice that, is some of these steps, we are simulataneously applying the same elementary move to two distinct subgraphs. The strategy for proving the statement in the general case can be easily extrapolated from here, and therefore we will not discuss it in any further detail. 

\begin{figure}[H]
\includegraphics[height=1in]{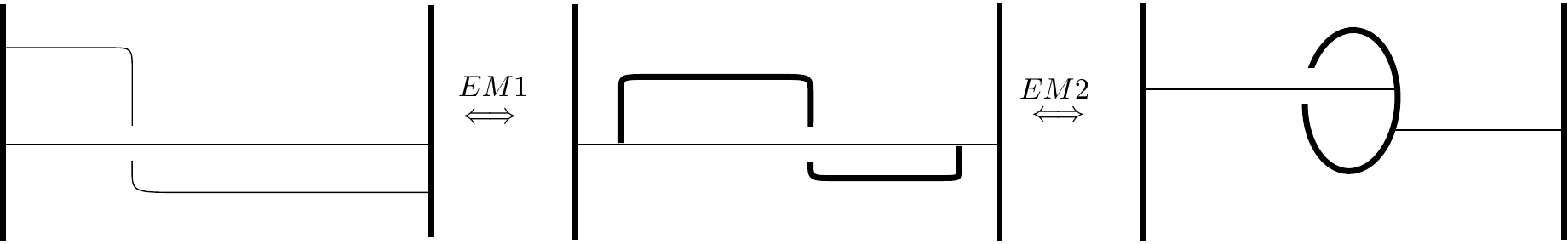}
\caption{The case $n=2$.}
\label{fig:n2}
\end{figure}

\begin{figure}[H]

\includegraphics[height=2.5in]{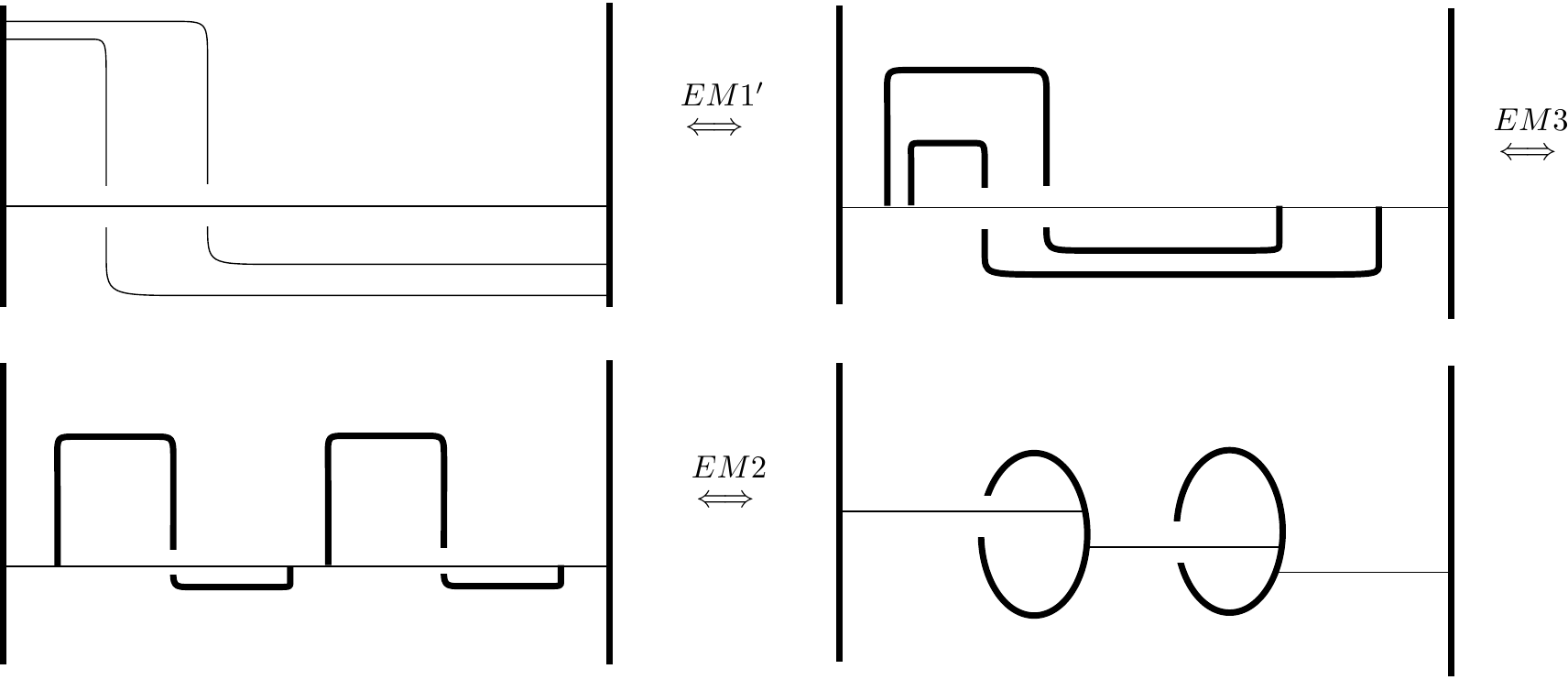}
\caption{The case $n=3$.}
\label{fig:n3}
\end{figure}

\end{document}